\documentclass[10pt]{amsart}
\usepackage{enumerate,amsmath,amssymb,latexsym,
amsfonts, amsthm, amscd}

\theoremstyle{plain}

\newtheorem{theorem}{Theorem}

\numberwithin{equation}{section}

\newcommand{\ra}{\rightarrow}

\newcommand{\E}{\mathbb{E}}
\newcommand{\F}{\mathcal{F}}

\addtocounter{section}{-1}

\begin{document}

\title {`Twisted Duality' in the ${\bf C}^*$  Clifford Algebra}

\date{}

\author[P.L. Robinson]{P.L. Robinson}

\address{Department of Mathematics \\ University of Florida \\ Gainesville FL 32611  USA }

\email[]{paulr@ufl.edu}

\subjclass{} \keywords{}

\begin{abstract}

Let $V$ be a real inner product space and $C[V]$ its ${\rm C}^*$ Clifford algebra. We prove that if $Z$ is a subspace of $V$ then  $C[Z^{\perp}]$ coincides with the supercommutant of $C[Z]$ in $C[V]$. 

\end{abstract}

\maketitle

\bigbreak

\section{Introduction}

Let $V$ be a real vector space on which $( \cdot | \cdot )$ is an inner product. Denote by $C(V)$ the associated complex Clifford algebra: thus, $C(V)$ is a unital associative complex algebra containing and generated by its real subspace $V$ subject to the Clifford relations 
\[ (\forall v \in V) \; v^2 = (v | v) {\bf 1}
\]
and it admits a unique involution $^*$ for which each element of $V$ is selfadjoint. This involutive algebra carries a natural norm $|| \cdot ||$ that satisfies the ${\rm C}^*$ property; the ${\rm C}^*$-algebra obtained upon its completion is the ${\rm C}^*$ Clifford algebra $C[V]$. 

\medbreak 

For our purposes, it is important to note that these Clifford algebras are naturally $\mathbb{Z}_2$-graded: they are {\it superalgebras}. Explicitly, the complex algebra $C(V)$  carries a unique automorphism $\gamma$ that restricts to $V \subseteq C(V)$ as $- {\rm Id}$; accordingly, $C(V)$ has an eigendecomposition
\[ C(V) = C(V)_+ \oplus C(V)_-
\] 
in which 
\[  C(V)_+ = \ker (\gamma - {\rm Id})
\]
is the even Clifford algebra and 
\[ C(V)_- = \ker (\gamma + {\rm Id})
\]
is the odd subspace. The grading automorphism $\gamma$ extends continuously to the ${\rm C}^*$ Clifford algebra $C[V]$ and thereby yields a corresponding eigendecomposition 
\[ C[V] = C[V]_+ \oplus C[V]_-. 
\]

We remark that the closure of $V$ in $C[V]$ is a copy of its Hilbert space completion, for which also $C[V]$ serves as the ${\rm C}^*$ Clifford algebra; accordingly, we may and shall assume that $V$ is a real Hilbert space. Likewise, we may and shall take subspaces of $V$ that appear below to be closed. 

\medbreak 

Thus, let $Z$ be a (closed) subspace of $V$ and $Z^{\perp}$ its orthocomplement, so that 
\[ V = Z \oplus Z^{\perp}
\]
is an orthogonal direct sum decomposition of the real Hilbert space $V$. Regarding $C[V]$ as a superalgebra, the {\it supercommutant} $C[Z]'$ of $C[Z]$ within $C[V]$ is 
\[ C[Z]' = C[Z]'_+ \oplus C[Z]'_-
\] 
with even part 
\[  C[Z]'_+ = \{ a \in C[V]_+ | (\forall b \in C[Z] ) \; b a = a b\}
\]
and odd part 
\[ C[Z]'_-  = \{ a \in C[V]_- |  (\forall b \in C[Z] ) \; b a =a \gamma(b)\}. 
\]
Thus: {\it even} elements of $C[Z]'$ commute with each element of $C[Z]$; {\it odd} elements of $C[Z]'$ commute with even elements of $C[Z]$ and anticommute with odd elements of $C[Z]$.  As elements of $V$ itself are odd, it follows that 
\[ C[Z]' = \{ a \in C[V] | (\forall z \in Z) \; za = \gamma(a) z \}. 
\]

Our purpose in this note is to prove the following version of `twisted duality'.  

\begin{theorem} 
$C[Z]' = C[Z^{\perp}]$. 
\end{theorem} 

For the origins and development of `twisted duality' see [2] [6] [3] [1]; several proofs of `twisted duality' for the plain complex Clifford algebra in line with the present account are presented in [5]; and [4] is a convenient reference for the theory of Clifford algebras. 

\section{Conditional expectations and twisted duality}

Our approach to twisted duality will be via conditional expectations: we shall construct a conditional expectation 
\[ \E_Z : C[V] \ra C[Z^{\perp}]
\]
that fixes the subalgebra $C[Z]'$ pointwise; this establishes the inclusion $C[Z]' \subseteq C[Z^{\perp}]$ while the reverse inclusion $C[Z^{\perp}] \subseteq C[Z]'$ is immediate from the (linearized) Clifford relations. 

\medbreak 

We shall find it convenient to denote by $\F (V)$ the set of all finite-dimensional subspaces of $V$ directed by inclusion and to observe that 
\[ C(V) = \bigcup_{M \in \F (V)} C(M). 
\]

\medbreak 

To begin, let $u \in V$ be a unit vector and $u^{\perp} \subseteq V$ its orthogonal space. The Clifford algebra of $V$ decomposes naturally as 
\[ C(V) = C(u^{\perp}) \oplus u \: C(u^{\perp}). 
\]
To see this, let  $c \in C(V)$ and choose $M \in \F (V)$ containing $u$ so that $c \in C(M)$; augment $u$ to an orthonormal basis for $M$ and expand $c$ relative to this basis. Suppose that in this decomposition,
\[ c = a + u b
\] 
with $a, b \in C(u^{\perp})$. Then $\gamma(a) u = ua$ and $\gamma(b) u = ub$ on account of the (linearized) Clifford relations, so that 
\[ u \gamma(c) u = u \gamma(a) u + u \gamma(u b) u =  a - u b
\]
because the unit vector $u$ has square ${\bf 1}$ and is odd. It follows that the projector 
\[ E_u : C(V) \ra  C(u^{\perp}) \subseteq C(V)
\]
of $C(V)$ on $C(u^{\perp})$ along $u C(u^{\perp})$ is given by 
\[ (\forall c \in C(V)) \; E_u (c) = \frac{1}{2} (c + u \gamma(c) u). 
\]

Let $M \in \F (V)$ be a finite-dimensional subspace of $V$. If $\{u_1, \dots , u_m \}$ is an orthonormal basis for $M$ then the projectors $E_{u_1} , \dots , E_{u_m}$ commute, as is shown by direct calculation using the Clifford relations; the product
\[ E_M = E_{u_1} \circ \dots \circ E_{u_m}
\]
projects $C(V)$ on $C(u_{u_1}^{\perp}) \cap \cdots \cap C(u_{u_m}^{\perp}) = C(M^{\perp})$. 

\medbreak 

Let $Z$ be an arbitrary (closed) subspace of $V$. Let $c \in C(V)$ and choose $M \in \F (V)$ so that $c \in C(M)$. Let $X \in \F (Z)$ be the orthogonal projection of $M$ on $Z$ and $Y \in \F (Z^{\perp})$ its orthogonal projection on $Z^{\perp}$; thus $M \subseteq X \oplus Y$. Apply $E_X$ to $c \in C(X \oplus Y)$ to obtain 
\[ E_X (c) \in C((X \oplus Y) \cap X^{\perp}) = C(Y) \subseteq C(Z^{\perp});
\]
 if $u \in Z$ is any unit vector then $u E_X (c) u = E_X (c)$ by the Clifford relations, so 
\[ E_u \circ E_X (c) = E_X (c). 
\]
Now, if $N \in \F (Z)$ is any finite-dimensional subspace of $Z$ containing $X$ then taking the product as $u$ runs over an orthonormal basis for $N \cap X^{\perp}$ reveals that 
\[ E_N (c) = E_X (c).
\]
This proves that the net $(E_N (c) | N \in \F (Z))$ is eventually constant and hence converges. 

\medbreak 

Thus, we construct a linear map 
\[E_Z : C(V) \ra C(Z^{\perp}) \subseteq C(V)
\]
 by the rule 
\[ (\forall c \in C(V)) \; E_Z (c) = \lim_{N \uparrow \F (Z)} E_N (c). 
\]

\begin{theorem} 
If $Z$ is a (closed) subspace of $V$ then $E_Z$ is contractive: 
\[ (\forall c \in C(V)) \; \| E_Z (c) \| \leqslant \| c \|
\]
and has the (conditional expectation) property: 
\[(\forall c \in C(V))(\forall \ell, r \in C(Z)') \;  E_Z (\ell c r) = \ell\: E_Z (c) \: r. 
\]
\end{theorem} 

\begin{proof} 
The norm here is the natural one, completion of $C(V)$ relative to which yields $C[V]$. If $u \in V$ is a unit vector then $\| v \| = 1$ so that if $c \in C(V)$ then $\| E_u (c) \| \leqslant \| c \|$ because $\gamma$ is isometric. Taking the product over an orthonormal basis shows that $E_N$ is contractive when $N \in \F (V)$. That $E_Z$ itself is contractive follows immediately. If $u \in Z$ is a unit vector then  $u \gamma(\ell) = \ell u$ and $\gamma(r) u = u r$ so that $u \gamma(\ell c r) u = \ell u \gamma(c) u r$ and therefore $E_u (\ell c r) = \ell E_u (c) r$. Taking the product over an orthonormal basis proves the conditional expectation property for $E_N$ when $N$ is a finite-dimensional subspace of $Z$. The conditional expectation property for $E_Z$ itself follows upon taking $M \in \F (V)$ so large that $c, \ell, r$ lie in  $C(M)$ and taking $N \in \F (Z)$ to contain the orthogonal projection of $M$ on $Z$.  
\end{proof}

\medbreak 

We remark that the purely algebraic conditional expectation $E_Z : C(V) \ra C(Z^{\perp})$ fixes the supercommutant $C(Z)' \subset C(V)$ pointwise: if $a \in C(Z)'$ then $E_Z (a) = a$ as follows at once from Theorem 2 by taking (say) $\ell = a$ and $c = r = {\bf 1}$. In particular, this implies that $C(Z)' \subseteq C(Z^{\perp})$; the Clifford relations yield the reverse $C(Z^{\perp}) \subseteq C(Z)'$. In this way, we establish Theorem 1 for the plain complex Clifford algebra C(V). 

\medbreak 

As the linear map $E_Z : C(V) \ra C(Z^{\perp}) \subseteq C(V) \subseteq C[V]$ is a contraction, it extends continuously to a contraction
\[ \E_Z : C[V] \ra C[Z^{\perp}].
\]
Note that if $N \in \F (V)$ has orthonormal basis $\{ u_1, \dots , u_n \}$ then the obvious factorization 
\[ \E_N = \E_{u_1} \circ \cdots \circ \E_{u_n}
\]
holds by continuous extension of its counterpart on the plain complex Clifford algebra, where if $u \in V$ is a unit vector then 
\[\E_u : C[V] \ra C[V] : c \mapsto \frac{1}{2}(c + u \gamma(c) u). 
\]

\begin{theorem} 
If $Z$ is a (closed) subspace of $V$ and $c \in C[V]$ then the net 
\[ (\E_N (c) | N \in \F (Z) )
\]
converges to $\E_Z (c)$: 
\[ \E_Z (c) = \lim_{N \uparrow \F (Z)} \E_N (c). 
\]
\end{theorem}

\begin{proof} 
Let $\varepsilon > 0$. Choose $c_{\varepsilon} \in C(V)$ so that $\| c - c_{\varepsilon} \| \leqslant \varepsilon$. Choose $M_{\varepsilon} \in \F (V)$ so that $c_{\varepsilon} \in C(M_{\varepsilon})$ and let $X_{\varepsilon} \in \F (Z)$ be the orthogonal projection of $M_{\varepsilon}$ on $Z$. Now, let $N \in \F (Z)$ contain $X_{\varepsilon}$: as $\E_Z$ and $\E_N$ are contractions, $\| \E_Z (c) - \E_Z (c_{\varepsilon}) \| \leqslant \varepsilon$ and $\| \E_N (c_{\varepsilon}) - \E_N (c) \| \leqslant \varepsilon$; also, $\E_Z (c_{\varepsilon}) = E_Z(c_{\varepsilon}) = E_N (c_{\varepsilon}) = \E_N(c_{\varepsilon})$. According to the triangle inequality, it follows that if $\F (Z) \ni N \supseteq X_{\varepsilon}$ then $\| \E_Z (c) - \E_N (c) \| \leqslant 2 \varepsilon$ and the proof is complete. 
\end{proof}

We are now able to establish Theorem 1 in full. 

\medbreak 
\noindent 
{\it Proof of Theorem 1}. As noted previously, the Clifford relations yield $C[Z^{\perp}] \subseteq C[Z]'$. We need only prove the reverse inclusion, so let $c \in C[Z]'$. Recall the formulae displayed prior to Theorem 3: if $u \in Z$ is a unit vector, then $\E_u (c) = c$; taking the product as $u$ runs over an orthonormal basis, it follows that if $N \in \F (Z)$ then $\E_N (c) = c$. Finally, Theorem 3 allows us to pass to the limit as $N$ runs over $\F (Z)$ to conclude that $c = \E_Z (c) \in C[Z^{\perp}]$.

\qed

We announced that $\E_Z : C[V] \ra C[Z^{\perp}]$ would be a conditional expectation; indeed it is. Let $c \in C[V]$ and let $\ell, r \in C[Z^{\perp}] = C[Z]'$; choose sequences $C(V) \ni c_n \ra c$, $C(Z^{\perp}) \ni \ell_n \ra \ell$ and $C(Z^{\perp}) \ni r_n \ra r$. The conditional expectation property of $E_Z$ in Theorem 2 justifies the middle step in 
\[ \E_Z (\ell_n c_n r_n) = E_Z (\ell_n c_n r_n) = \ell_n E_Z (c_n) r_n = \ell_n \E_Z (c_n) r_n 
\]
whence continuity of $\E_Z$ and passage to the $n \ra \infty$ limit yield the required identity
\[ \E_Z( \ell c r) = \ell \: \E_Z (c) \:  r.
\]
Further, $\E_Z$ preserves the involution $^*$ and in fact preserves positivity. Let $c \in C(V)$: if $u \in V$ is a unit vector then 
\[ E_u (c^* c) = \frac{1}{2} (c^* c +  (\gamma(c) u)^*  (\gamma(c) u) )
\]
is a convex combination of terms $d^* d$ for $d \in C(V)$ so the same is true of $E_N (c^* c)$ whenever $N \in \F (Z)$ and therefore true of $E_Z (c^* c)$; by continuity, it follows that $\E_Z (c^* c) \geqslant 0$ whenever $c \in C[V]$. Of course, $\E_Z$ is idempotent. 

\medbreak 

We close by offering what is perhaps an eccentric application of `twisted duality'. 

\begin{theorem} 
If $\{ Z_{\lambda} | \lambda \in \Lambda \}$ is any family of (closed) subspaces of $V$ then 
\[ \bigcap_{\lambda \in \Lambda} C[Z_{\lambda}] = C[ \bigcap_{\lambda \in \Lambda} Z_{\lambda}]. 
\]
\end{theorem} 

\begin{proof} 
Only $\subseteq$ is in question, so let $c \in \bigcap_{\lambda \in \Lambda} C[Z_{\lambda}]$. If $\lambda \in \Lambda$ then $c \in  C[Z_{\lambda}] \subseteq C[Z_{\lambda}^{\perp}]'$ by (easier direction) `twisted duality' so that if also $v \in Z_{\lambda}^{\perp}$ then $v c = \gamma(c) v$. By linearity, it follows that 
\[ v \in \sum_{\lambda \in \Lambda} Z_{\lambda}^{\perp} \; \Rightarrow \; v c = \gamma(c) v
\]
whence continuity yields 
\[ v \in \overline{\sum_{\lambda \in \Lambda} Z_{\lambda}^{\perp}}  \; \Rightarrow \; v c = \gamma(c) v. 
\]
Recalling the Hilbert space identity 
\[ \overline{\sum_{\lambda \in \Lambda} Z_{\lambda}^{\perp}} = (\bigcap_{\lambda \in \Lambda} Z_{\lambda})^{\perp}
\]
we conclude from (harder direction) `twisted duality' that
\[ c \in C[\overline{\sum_{\lambda \in \Lambda} Z_{\lambda}^{\perp}}]' \subseteq C[(\overline{\sum_{\lambda \in \Lambda} Z_{\lambda}^{\perp}})^{\perp}] = C[\bigcap_{\lambda \in \Lambda} Z_{\lambda}].
\]
\end{proof}

\bigbreak 
\noindent 
\begin{center}
REFERENCES
\end{center}
\medbreak 
\noindent
[1] H. Baumgartel, M. Jurke and F. Lledo, {\it Twisted duality of the CAR-Algebra}, J. Math. Phys. {\bf 43 (8)} (2002) 4158-4179. 
\medbreak 
\noindent 
[2] S. Doplicher, R. Haag and J. Roberts, {\it Fields, observables and gauge transformations. I}, Comm. Math. Phys. {\bf 13} (1969) 1-23. 
\medbreak 
\noindent
[3] J. J. Foit, {\it Abstract Twisted Duality for Quantum Free Fermi Fields}, Publ. RIMS, Kyoto Univ. {\bf 19} (1983) 729-741. 
\medbreak 
\noindent 
[4] R. J. Plymen and P. L. Robinson, {\it Spinors in Hilbert Space}, Cambridge Tracts in Mathematics {\bf 114} (1994). 
\medbreak 
\noindent 
[5] P. L. Robinson, {\it `Twisted Duality' for Clifford Algebras}, arXiv:1407.1420v1 (2014). 
\medbreak 
\noindent 
[6] S.J. Summers, {\it Normal Product States for Fermions and Twisted Duality for CCR- and CAR-Type Algebras with Application to the Yukawa$_2$ Quantum Field Model}, Comm. Math. Phys. {\bf 86} (1982) 111-141. 

\end{document}